\documentclass[11pt, letterpaper, oneside]{article}

\usepackage[T1]{fontenc}
\usepackage{ae}
\usepackage{aecompl}
\usepackage[latin1]{inputenc}
\usepackage{latexsym}
\usepackage{amssymb}
\usepackage{graphicx}
\usepackage{psfrag}
\usepackage{amsmath}
\usepackage{amsthm}
\usepackage{amsfonts}
\usepackage[all]{xy}
\usepackage[british]{babel}
\usepackage{soul}
\usepackage[paper = letterpaper, left = 2.7cm, right = 2.7cm, headsep = 8mm, footskip = 7mm,
top = 25mm, bottom = 25mm]{geometry}

\newcommand*{\definiere}{\mathrel{\mathop:}=}

\newcommand{\cyclic}{\mathop{\kern0.9ex{{+}\kern-2.10ex\raise-0.20
      ex\hbox{\Large\hbox{$\circlearrowright$}}}}\limits}
\newcommand{\acts}{\mbox{ \raisebox{0.26ex}{\tiny{$\bullet$}} }}

\def\N{\ifmmode{\mathbb N}\else{$\mathbb N$}\fi}
\def\Z{\ifmmode{\mathbb Z}\else{$\mathbb Z$}\fi}
\def\Q{\ifmmode{\mathbb Q}\else{$\mathbb Q$}\fi}
\def\R{\ifmmode{\mathbb R}\else{$\mathbb R$}\fi}
\def\C{\ifmmode{\mathbb C}\else{$\mathbb C$}\fi}
\def\K{\ifmmode{\mathbb K}\else{$\mathbb K$}\fi}
\def\P{\ifmmode{\mathbb P}\else{$\mathbb P$}\fi}
\def\g{\ifmmode{\mathfrak g}\else {$\mathfrak g$}\fi}
\def\h{\ifmmode{\mathfrak h}\else {$\mathfrak h$}\fi}
\def\a{\ifmmode{\mathfrak a}\else {$\mathfrak a$}\fi}
\def\k{\ifmmode{\mathfrak k}\else {$\mathfrak k$}\fi}
\def\p{\ifmmode{\mathfrak p}\else {$\mathfrak p$}\fi}
\def\b{\ifmmode{\mathfrak b}\else {$\mathfrak b$}\fi}
\def\n{\ifmmode{\mathfrak n}\else {$\mathfrak n$}\fi}
\def\m{\ifmmode{\mathfrak m}\else {$\mathfrak m$}\fi}
\def\t{\ifmmode{\mathfrak t}\else {$\mathfrak t$}\fi}
\def\O{\ifmmode{\mathcal{O}}\else {$\mathcal{O}$}\fi}
\def\W{\ifmmode{\mathcal{W}}\else {$\mathcal{W}$}\fi}
\def\id{{\rm id}}

\def\hq{/\hspace{-0.12cm}/}
\def\kleinematrix#1,#2,#3,#4,{\begin{pmatrix}#1 & #2 \\ #3 & #4
  \end{pmatrix}}

\DeclareMathOperator{\Ad}{Ad}

\DeclareMathOperator{\Lie}{Lie}

\newtheorem{thm}{Theorem}[section]
\newtheorem{prop}[thm]{Proposition}
\newtheorem{Defi}[thm]{Definition}
\newtheorem{lemma}[thm]{Lemma}
\newtheorem{cor}[thm]{Corollary}
\newtheorem{Exs}[thm]{Examples}
\newtheorem{Ex}[thm]{Example}
\newtheorem{Rems}[thm]{Remarks}
\newtheorem{Rem}[thm]{Remark}

\newtheorem{propdef}[thm]{Proposition and Definition}

\newenvironment{rem}   {\begin{Rem}\em}{\end{Rem}}

\newenvironment{defi}  {\begin{Defi}\em}{\end{Defi}}
\newenvironment{ex}  {\begin{Ex}\em}{\end{Ex}}

\begin{document}
\author{Daniel Greb\thanks{The first author is supported by a Promotionsstipendium of the
  Studienstiftung des deutschen Volkes and by SFB/TR 12 of the DFG.} \and Peter Heinzner\thanks{The
  second author is partly supported by SFB/TR 12 of the DFG.}}

\title{K\"ahlerian Reduction in Steps}
\date{}

\maketitle
\begin{abstract}
We study Hamiltonian actions of compact Lie groups $K$ on  K\"ahler manifolds which extend to a
holomorphic action of the complexified group $K^\C$. For a closed normal subgroup $L$ of $K$ we show
that the K\"ahlerian reduction with respect to $L$ is a stratified Hamiltonian K\"ahler
$K^\C/L^\C$-space whose K\"ahlerian reduction with respect to $K/L$ is naturally isomorphic to the
K\"ahlerian reduction of the original manifold with respect to $K$.
\end{abstract}

\section{Introduction}
Reduction of variables for physical systems with symmetries is a
fundamental concept in classical Hamiltonian dynamics. It is based on Noether's principle that every $1$-parameter group of symmetries of a physical system corresponds to a constant of motion. The mathematical formalisation is known as the Marsden-Weinstein-Reduction or \emph{symplectic
reduction}. Consider a symplectic manifold $X$ with a smooth symplectic form $\omega$ and a smooth action of a Lie group $K$ on $X$ by $\omega$-isometries. Assume that there exists a smooth map $\mu$ from $X$ into the dual space $\mathfrak{k}^*$ of the Lie algebra $\mathfrak{k}$ of $K$ such that
\begin{itemize}
 \item[a)] for every $\xi \in \mathfrak{k}$, the function $\mu^\xi: X \to \R$, $\mu^\xi(x) = \mu(x)(\xi)$, fulfills $d\mu^\xi = \imath_{\xi_X}\omega$, where $\xi_X$ denotes the vector field on $X$ induced by the action of $K$ on $X$ and $\imath_{\xi_X}$ denotes contraction with respect to $\xi_X$,
\item[b)] the map $\mu:X \to \mathfrak{k}^*$ is equivariant with respect to the action of $K$ on $X$ and the coadjoint representation $\Ad^*$ of $K$ on $\mathfrak{k}^*$, i.e.\ for all $x\in X$ and for all $\k \in K$ we have $\mu(k \acts x) = \Ad^*(k)(\mu(x))$.
\end{itemize}
The map $\mu: X \rightarrow \mathfrak{k}^*$ is called a ($K$-equivariant) \emph{momentum map} and the action of $K$ on $X$ is called \emph{Hamiltonian}.

The motion of a classical particle is given by a Hamiltonian function $H$. More precisely, let $H$ be a smooth function on $X$. The time evolution of the underlying physical system is described by the flow of the vector field $V_H$ associated to the Hamiltonian $H$ via the equation $dH = \imath_{V_H}\omega$. A Hamiltonian system with symmetries is a system $(X,\omega, K, \mu, H)$ where $H$ is invariant with respect to the $K$-action. In this case it follows from a) that 
\[d\mu^\xi (V_H) = \omega (\xi_X, V_H) = -dH(\xi_X) = 0 \]
holds for all $\xi \in \mathfrak{k}$. This implies Noether's principle in the following geometric formulation: every component $\mu^\xi$ of the momentum map is a constant of motion for the system described by $H$. 

The previous considerations imply that level sets of $\mu$ are
invariant under the flow of $V_H$. In many cases questions can be
organised so that $\mu^{-1}(0)$ is the momentum fibre of interest. By b), the group $K$ acts on the level set $\mu^{-1}(0)$. Let us first consider this action on an infinitesimal level. If we fix $x_0 \in \mathcal{M} \definiere \mu^{-1}(0)$, it follows from a) that $\text{ker} (d\mu(x_0)) = (\mathfrak{k}\acts x_0)^{\perp_\omega}$, where $\mathfrak{k}\acts x_0 = \{\xi_X(x_0) \,|\, \xi \in \mathfrak{k}\}$ and $^{\perp_\omega}$ denotes perpendicular with respect to $\omega$. The optimal situation appears if $\mu$ has maximal rank at $x_0$. In this case, $\mathcal{M}$ is smooth at $x_0$ and $T_{x_0}\mathcal{M}=\text{ker} (d\mu(x_0)) = (\mathfrak{k}\acts x_0)^{\perp_\omega}$ holds. It follows that $\mathcal{M}$ is a coisotropic $K$-stable submanifold of $X$ and the symplectic form $\omega_{x_0}$ on $T_{x_0}\mathcal{M}$ induces a symplectic form $\tilde \omega_{x_0}$ on $T_{x_0}\mathcal{M}/T_{x_0}\mathcal{M}^{\perp_\omega} = T_{x_0}\mathcal{M}/\mathfrak{k}\acts x_0$. These observations imply that once the space $\mathcal{M}/K$ is smooth, it will be a symplectic manifold. This is the content of the \emph{Marsden-Weinstein-Theorem} (see \cite{SympRed}):

\emph{If $K$ acts freely and properly on $\mathcal{M}$, the quotient $\mathcal{M}/K$ is a symplectic manifold whose symplectic form $\tilde \omega$ is characterised via the equation}
\[ \pi^* \tilde \omega = i^*_{\mathcal{M}}\omega.\]
Here, $\pi: \mathcal{M} \rightarrow \mathcal{M}/K$ denotes the quotient map and $i_{\mathcal{M}}: \mathcal{M} \rightarrow X$ is the inclusion.
Furthermore, the restriction of the $K$-invariant Hamiltonian $H$ to $\mathcal{M}$ induces a smooth function $\tilde H$ on $\mathcal{M}/K$. The Hamiltonian system on $(\mathcal{M}/K, \tilde \omega)$ associated to $\tilde H$ captures the essential (symmetry-independent) properties of the original $K$-invariant system that was given by $H$.

The Marsden-Weinstein construction is natural in the sense that it can
be done in steps. This means that for a normal closed subgroup $L$ of
$K$, the restricted momentum map $\mu_L: X \rightarrow \mathfrak{l}^*$
is $K$-equivariant, the induced $K$-action on $\mu^{-1}_L(0)/L$ is
Hamiltonian with momentum map $\tilde \mu$ induced by $\mu$ and the
symplectic reduction $\tilde \mu^{-1}(0)/K$ is symplectomorphic to
$\mathcal{M}/K$.

Removing the restrictive regularity assumptions of \cite{SympRed}, it
is proven in \cite{StratifiedSymplectic} that symplectic reduction can
be carried out for general group actions of compact Lie groups yielding stratified symplectic quotient spaces $\mathcal{M}/K$, i.e.\ stratified spaces where all strata are symplectic manifolds. The paper \cite{Extensionofsymplectic} proposed an approach to this singular symplectic reduction based on embedding symplectic manifolds into K\"ahler manifolds and a K\"ahler reduction theory for K\"ahler manifolds. Roughly speaking, $\mathcal{M}/K$ is realised as a locally semialgebraic subspace of a K\"ahler space $Q$ and the symplectic structure on $\mathcal{M}/K$ is given by restriction. Inspecting the proofs one sees that the construction of \cite{Extensionofsymplectic} is compatible with reduction in steps.

Our interest here is to study the problem of reduction in steps in a K\"ahlerian context using techniques related to the complex geometry and invariant theory for the complexification $K^\C$ of a compact Lie group $K$.

Holomorphic actions of compact groups $K$ on complex spaces $X$ very often
extend to holomorphic actions of the complexification $K^\C$,
at least in the sense that there exists an open $K$-equivariant
embedding $X \hookrightarrow X^c$ of $X$ into a holomorphic $K^\C$-space
$X^c$ (see \cite{HeinznerIannuzzi}). In this note, we consider the case where $K^\C$ already acts on
$X$. More precisely, we consider K\"ahler manifolds $X$ with a Hamiltonian action of a compact Lie
group $K$ that extends to a holomorphic action of the complexification
$K^\C$. In this setup, an extensive quotient theory has been
developed. See Section \ref{groupsonKaehler} for a survey of the results used in this note. Due to the action of the complex group
$K^\C$, it is possible to introduce a notion of semistability. The reduced space $\mathcal{M}/K$ carries a complex structure
given by holomorphic invariant theory for the action of $K^\C$ on the set of semistable points. The reduced symplectic structure is
compatible with this complex structure, i.e. $\mathcal{M}/K$ is a
K\"ahler space. Since the quotient $\mathcal{M}/K$ in general will be singular, the K\"ahler structure is locally given by continuous strictly plurisubharmonic functions that are smooth along the natural
stratification of $\mathcal{M}/K$ given by the orbits of the action of $K$ on
$\mathcal{M}$. We emphasize at this point that both the complex
structure and the K\"ahler structure are defined globally, i.e.\ across
the strata. We call $\mathcal{M}/K$ the \emph{K\"ahlerian reduction} of
$X$ by $K$. This reduction theory works more generally for stratified
K\"ahlerian complex
spaces, however, in order to reduce technical difficulties and not to
obscure the main principles at work, we will restrict our attention to actions on complex manifolds. Nevertheless, we also have to consider complex spaces which appear as quotients by normal subgroups of $K^\C$.

Using the approach of \cite{Extensionofsymplectic} and the quotient theory for complex-reductive group actions, we show that K\"ahlerian reduction can be done in
steps. More precisely, we prove that the K\"ahlerian reduction of $X$
by a normal subgroup $L$ of $K$ is a stratified Hamiltonian K\"ahler $K^\C$-space
and that its K\"ahlerian reduction is isomorphic to the reduction of
$X$ by $K$ (see Theorem \ref{main}). This shows that reduction in steps respects the K\"ahler
geometry. It also exhibits the class of K\"ahlerian stratified spaces
as natural for a K\"ahlerian reduction theory. Furthermore, we look carefully at the stratifications obtained on the various quotient spaces.
\subsection{Reductive group actions on K\"ahler spaces}\label{groupsonKaehler}
As we have noted above, symplectic reduction of K\"ahler manifolds
yields spaces that are endowed with a complex structure. This is
due to a close relation between the quotient theory of a compact Lie
group $K$ and the quotient theory for its complexification $K^\C$
which we will now explain.

In this paper a complex space refers to a reduced complex space with countable topology. If $G$ is a
Lie group, then a \emph{complex $G$-space $Z$} is a complex space with a real-analytic action $G
\times Z \to Z$ which for fixed $g \in G$ is holomorphic. For a complex Lie group $G$ a
\emph{holomorphic $G$-space $Z$} is a complex $G$-space such that the action $G \times Z \to Z$ is
holomorphic.

We note that given a
compact real Lie group $K$, there exists a complex Lie group $K^\C$
containing $K$ as a closed subgroup with the following universal
property: given a Lie homomorphism $\varphi: K \to H$ into a complex
Lie group $H$, there exists a holomorphic Lie homomorphism $\varphi^\C:
K^\C \to H$ extending $\varphi$. 

If not only $K$ but also $K^\C$ acts holomorphically on a manifold $X$, it is natural
to relate the quotient theory of $K$ to the quotient theory of
$K^\C$. Due to the
existence of non-closed $K^\C$-orbits, in contrast to actions of $K$, actions of $K^\C$
will in general not give rise to reasonable orbit spaces. However, as we will see, it is often
possible to find an open subset $U$ of $X$ and a complex space $Y$
that parametrises closed $K^\C$-orbits in $U$. More precisely, we call a complex space $Y$ an
\emph{analytic Hilbert quotient} of $U$ by
the action of $K^\C$ if there exists a $K^\C$-invariant Stein holomorphic map
$\pi: U \rightarrow Y$ with $(\pi_*\mathcal{O}_U)^{K^\C} =
\mathcal{O}_Y$. Here, a \emph{Stein map} is a map such that inverse
images of Stein subsets are again Stein. It can be shown (see \cite{ReductionOfHamiltonianSpaces}, for example) that $Y$ is the quotient of $U$ by the
equivalence relation 
\[x \sim y \quad \quad \text{ if and only if }\quad \quad \overline{K^\C \acts x} \cap  \overline{K^\C
  \acts y} \neq \emptyset.
\]

Analytic Hilbert quotients are universal with respect to
$K^\C$-invariant holomorphic maps, i.e.\ given a $K^\C$-invariant
holomorphic map $\varphi: U \rightarrow Z$ into a complex space $Z$,
there exists a uniquely determined holomorphic map $\bar \varphi: Y
\rightarrow Z$ such that $\varphi = \bar \varphi \circ \pi$.  It follows that an analytic Hilbert quotient of $U$ by $K^\C$ is unique
up to biholomorphism once it exists. We denote it by $U\hq K^\C$. 

The theory of analytic Hilbert quotients is interwoven with the
K\"ahlerian quotient theory as follows (see
\cite{ReductionOfHamiltonianSpaces}): let $K$ be a compact Lie
group with Lie algebra $\mathfrak{k}$, $X$ a holomorphic $K^\C$-manifold. Assume that the action of $K$
is Hamiltonian with respect to a $K$-invariant K\"ahler form on $X$ with $K$-equivariant momentum map $\mu: X \to \mathfrak{k}^*$. In this situation, we call $X$ a \emph{Hamiltonian K\"ahler $K^\C$-manifold}. Let $\mathcal{M}= \mu^{-1}(0)$ and define
\[X(\mathcal{M}) \definiere \{x \in X \,|\, \overline{K^\C\acts x} \cap
\mathcal{M} \neq \emptyset \}. \]A point $x\in X(\mathcal{M})$ is
called \emph{semistable}. The set $X(\mathcal{M})$ is open in $X$ and an analytic Hilbert quotient $Q$ of
$X(\mathcal{M})$ by $K^\C$ exists. Each fibre of the quotient map $\pi: X(\mathcal{M}) \to Q$ contains a unique
closed $K^\C$-orbit which is the unique orbit of minimal dimension in
that fibre. This closed $K^\C$-orbit intersects $\mathcal{M}$ in a
unique $K$-orbit, $K\acts x$. The isotropy group $(K^\C)_x$ of points
$x \in \mathcal{M}$ is
complex-reductive and equal to the complexification of $K_x$. The analytic Hilbert quotient $X(\mathcal{M})\hq K^\C$ is related to the K\"ahlerian
reduction $\mathcal{M}/K$ by the following fundamental commutative diagram
\begin{equation}\label{fundamental}
\begin{xymatrix}{
\mathcal{M} \ar@{^{(}->}[r]^{\iota}\ar[d]_{\pi_K}& X(\mathcal{M})\ar[d]^{\pi}  \\
\mathcal{M}/K  \ar[r]_{\simeq}^{\bar \iota}   & Q.}
\end{xymatrix}
\end{equation}
Here, the induced map $\bar \iota$ is a homeomorphism. Hence, the
symplectic  reduction $\mathcal{M}/K$ has a complex structure induced
via the homeomorphism $\bar \iota$. The inverse of $\bar \iota$ is induced via a retraction
$\psi: X(\mathcal{M}) \to \mathcal{M}$ that is related to the stratification of
$X$ via the gradient flow of the norm square of the momentum map (see
\cite{NeemanTopology}, \cite{KirwanCohomology}, and \cite{GerrySurvey}). 

Using this two-sided picture of the quotient it can be shown using the techniques of
\cite{Extensionofsymplectic} that $\mathcal{M}/K$ carries a natural
K\"ahlerian structure that is smooth along the \emph{orbit type
  stratification}.
\subsection{Stratifying holomorphic $G$-spaces}\label{stratifications}
Invariant stratifications are a powerful
tool in the study of group actions and their quotient spaces (see
\cite{LunaSlice} and \cite{GerryHomotopies}). We will recall the
definitions and basic properties of these stratifications. 

\begin{defi}
A \emph{complex stratification} of a complex space $X$ is a countable,
locally finite covering of $X$ by disjoint subspaces (the so called \emph{strata}) $\mathbb{S} =
(S_\gamma)_{\gamma \in \Gamma}$ with the following properties
\begin{enumerate}
\item each stratum $S_\gamma$ is a
locally closed submanifold of $X$ that is Zariski-open in its
closure,
\item the boundary $\partial S_\gamma = \bar S_\gamma
\setminus S_\gamma$ of each stratum $S_\gamma$ is a union of strata of lower dimension.
\end{enumerate}
\end{defi}
\begin{ex}
The singular set $X_{sing}$ of a complex space $X$ is a closed complex subspace of
smaller dimension. Iterating this procedure, i.e.\ considering the
singular set of $X_{sing}$, we obtain a natural
stratification on $X$. If a Lie group $G$ acts holomorphically on $X$, this stratification is $G$-invariant.
\end{ex}
We now consider stratifications related to group actions.
Let $G$ be a complex-reductive Lie group. Let $X$ be a holomorphic $G$-manifold such that the analytic Hilbert
quotient $\pi_G: X \rightarrow X\hq G$ exists. Let $p$ be a point in
$X\hq G$. The fibre $\pi_G^{-1}(p)$ over $p$ contains a unique closed
orbit. We denote this orbit by $C(p)$. We say that $p$ is of \emph{$G$-orbit-type $(G_1)$} if the stabilizer
of one (and hence any point) in $C(p)$ is conjugate to $G_1$ in $G$.

The next result follows from the holomorphic slice theorem (see \cite{HeinznerGIT}):
\begin{propdef}
Let $X$ be a holomorphic $G$-manifold such that the analytic Hilbert quotient $\pi_G: X
\rightarrow X\hq G$ exists. Then, each connected
component of the set of points of orbit type $(G_1)$ is a locally closed
manifold and the corresponding decomposition of $X\hq G$ is a stratification of $X$ which we call the \emph{orbit type stratification of $X\hq G$}.
\end{propdef}

We obtain a related $G$-invariant complex stratification of $X$ by stratifying the
preimage $\pi_G^{-1}(S_\gamma)$ of each orbit type stratum $S_\gamma
\subset X\hq G$ as a complex space.

Over a stratum $S_\gamma$, the structure of $X$ and of the quotient
$\pi_G: X \rightarrow X\hq G$ is particularly simple. More precisely,
let $S_\gamma$ be a stratum of $X\hq G$ and let $(G_0)$ be the
conjugacy class of isotropy groups corresponding to $S_\gamma$. Let us
assume for the moment that $X\hq G = S_\gamma$. Then, the slice theorem implies
that each point $q \in X\hq G$ has an open neighbourhood $U$
such that $\pi_G^{-1}(U)$ is $G$-equivariantly biholomorphic to $G
\times_{G_0}Z$, where $Z$ is a locally closed $G_0$-stable Stein submanifold of $X$. The union of the
closed orbits in $\pi_G^{-1}(U)$ is equal to $G
\times_{G_0}Z^{G_0} \cong G/G_0 \times Z^{G_0}$. Noticing that $U
\cong  (G\times_{G_0}Z)\hq G \cong Z\hq G_0 \cong Z^{G_0}$, we see
that the set of closed orbits in $\pi_G^{-1}(S_\gamma)$ is a smooth
fibre bundle over $S_\gamma$ with typical fibre $G/G_0$ and structure
group $G$.

If $X\hq G$ is irreducible, there exists a Zariski-open and dense stratum $S_{princ}$ in
$X\hq G$, called the \emph{principal stratum}. It corresponds to the
minimal conjugacy class $(G_0)$ of isotropy groups of closed orbits, i.e.\  if $x$ is any point
$X$ with closed $G$-orbit, $G_0$ is conjugate in $G$ to a subgroup of $G_x$.

In the Hamiltonian setup there is a stratification related to the
action of $K$ on the momentum zero fibre: let $X$ be a Hamiltonian K\"ahler $K^\C$-space. Assume that $X = X(\mathcal{M})$. Decompose the quotient $\mathcal{M}/K$ by orbit types for the action of $K$ on $\mathcal{M}$. It has been shown in a purely symplectic setup in \cite{StratifiedSymplectic} that this defines a stratification of $\mathcal{M}/K$.
Since $\mathcal{M}/K$ is homeomorphic to $X\hq K^\C$ (see Diagram \eqref{fundamental}) this 
defines a second decomposition of $X\hq K^\C$. However,
it can be shown that the two constructions yield the same result: given a stratum
$S$ corresponding to $K$-orbit type $(K_1)$, it coincides with one of
the connected components of the set of points in $X\hq K^\C$ with
orbit type $(K_1^\C)$ (see \cite{SjamaarSlices}).
\subsection{K\"ahler reduction}\label{KaehlerReduction}
In this section we recall the basic definitions necesarry for K\"ahlerian reduction theory. 
Due to the presence of singularities in the spaces that we will
consider, K\"ahler structures are defined in terms of strictly
plurisubharmonic functions.
\begin{defi}
Let $Z$ be a complex space. A continuous function $\rho: Z \to
\R$ is called \emph{plurisubharmonic}, if
for every holomorphic map $\varphi$ from the unit disc $D$ in $\C$ to $Z$, the pullback $\varphi^* (\rho)$ is subharmonic on
$D$, i.\ e.\ for each $0 < r < 1$  the \emph{mean
value inequality} \[\varphi^*(\rho) (0) \leq \frac{1}{2\pi} \int_0^{2\pi}\varphi^*(\rho)(r e^{i \theta}) \;d\theta \]holds.

A \emph{perturbation} of a continuous function $\rho: Z \to \R $ at a point $x \in Z$ is a function $\rho + f$, where $f$ is smooth and defined in
some neighbourhood $U$ of $x$. The function $\rho$ is said to be
\emph{strictly plurisubharmonic} if for every perturbation $\rho + f$ there exist $\varepsilon > 0 $ and a perhaps smaller neighbourhood
$V$ of $x$ such that $\rho + \varepsilon f$ is plurisubharmonic on $V$.
\end{defi}
\begin{rem}
If $Z$ is a complex manifold and $\rho: Z \rightarrow \R$ is smooth,
then $\rho$ is strictly plurisubharmonic if and only if its Levi form
$\frac{i}{2}\partial \bar \partial \rho$ is positive definite, i.e. if
it defines a K\"ahler form.
\end{rem}

\begin{defi}
A \emph{K\"ahler structure} on a complex space $Z$ is
given by an open cover $(U_j)$ of $Z$ and a family of strictly
plurisubharmonic functions $\rho _j: U_j
\rightarrow \R$ such that the differences $\rho_j -
\rho_k$ are \emph{pluriharmonic} on $U_{jk}\definiere U_j \cap U_k$ in the sense that there exists a
holomorphic function $f_{jk} \in \mathcal{O}(U_{jk})$ with $\rho_j -
\rho_k = Re(f_{jk})$. Two K\"ahler structures $(U_j, \rho_j)$ and $(\tilde{U}_k, \tilde{\rho}_k)$
are considered equal if there exists a common refinement $(V_l)$ of
$(U_j)$ and $(\tilde{U}_k)$ such that $(\rho_l - \tilde \rho_l)|_{V_l}$ is pluriharmonic for
every $l$.
\end{defi}
Again, we remark that the definition made above coincides with the
usual definition of a K\"ahler form on a complex manifold if all the
$\rho_j$'s are assumed to be smooth. For more information on strictly plurisubharmonic functions we refer to \cite{Varouchas2} and \cite{FornaessNarasimhan}.

We have seen that analytic Hilbert quotients of manifolds have a natural stratification by
orbit types. Taking into account this additional structure, we make
the following definition.
\begin{defi}
A complex space $Z$ is called a \emph{stratified K\"ahler space} if
there exists a complex stratification $\mathbb{S}=(S_\gamma)_{\gamma
  \in \Gamma}$ on $Z$ which is finer than the stratification of $Z$ as
a complex space and there exists a K\"ahler structure $\omega = (U_\alpha , \rho_\alpha)_{\alpha \in I}$
  on $Z$ such that $\rho_\alpha |_{S_\gamma \cap U_\alpha}$ is smooth.
\end{defi}
The following theorem is a special case of results proven in \cite{Extensionofsymplectic} and \cite{ReductionOfHamiltonianSpaces}:
\begin{thm}\label{reduction}
Let $K$ be a compact Lie group and $G= K^\C$ its complexification. Let $X$ be a Hamiltonian K\"ahler $K^\C$-manifold with $X=X(\mathcal{M})$. We denote the quotient map by $\pi_G: X \rightarrow X\hq G$. Let $\mathbb{S}^{X\hq G}_G$ be the orbit type stratification of $X\hq G$
defined above. Then there exists a K\"ahler structure
$\tilde{\omega}=(U_\alpha, \tilde{\rho}_\alpha)_{\alpha\in I}$ on $X\hq G$ with the following properties:
\begin{enumerate}
 \item the tripel $(X, \mathbb{S}^{X\hq G}_G, \tilde{\omega})$ is a stratified
K\"ahler space,
\item there exist smooth functions $\rho_\alpha: \pi_H^{-1}(U_\alpha) \rightarrow \R$ with $\omega|_{\pi_H^{-1}(U_\alpha)}= 2i \partial\bar\partial
(\rho_\alpha)$ such that the following equality holds: $\pi_G^*(\tilde \rho_\alpha )|_{\mathcal{M}\cap \pi_G^{-1}(U_\alpha)}=\rho_\alpha|_{\mathcal{M}\cap \pi_G^{-1}(U_\alpha)}$.
\end{enumerate}
\end{thm}
We recall the construction of the reduced K\"ahler structure.
With the notation of the previous theorem, the key technical result for the construction can be stated as follows:
\begin{prop}\label{organisation}
Let $x \in \mathcal{M}$. Then, there exists a
$\pi_G$-saturated neighbourhood $U$ of $x$ in $X$, such that the K\"ahler structure
$\omega$ is given by a smooth strictly plurisubharmonic function $\rho:
U \rightarrow \R$. The function $\rho$ is
an exhaustion along every fibre of $\pi_G$ that is contained in $U$. Furthermore, the restriction of the momentum map
$\mu$ to $U$ fulfills
\begin{equation}\label{murho}
\mu^\xi(x) = \mu_\rho^\xi(x) \definiere \left.\frac{d}{dt}\right|_{t=0}\rho(\exp(it\xi)\acts x)\quad \forall \xi \in
\mathfrak{k}, \forall x\in U.
\end{equation}
The set $\mathcal{M}\cap U$ coincides with the set of critical points
for the restriction of $\rho$ to fibres of the quotient map $\pi_G$.
\end{prop}
\begin{rem}
If a momentum map $\mu$ fulfills equation \eqref{murho} for some function $\rho$, we say that \emph{$\mu$ is associated to $\rho$} and we write $\mu = \mu_\rho$.
\end{rem}
The K\"ahler structure on the quotient $X(\mathcal{M})\hq G$ is constructed as follows: every point $y \in X\hq G$ has a neighbourhood
of the form $U\hq G$ which has the properties of Proposition \ref{organisation}. The restriction of $\rho$
to $\mathcal{M}\cap U$ induces a continuous function $\tilde{\rho}$
on $U\hq G$. We will see that it is strictly plurisubharmonic and smooth along the strata of $\mathbb{S}^{X\hq G}_G$. Let $S$ be a stratum. Let $Y$ be the set of closed $G$-orbits in $\pi_G^{-1}(S)$. This is smooth and contains $\mathcal{M}\cap
\pi^{-1}_G(S)$ as a smooth submanifold. We know that $\omega$ and hence
$\rho$ are smooth along $Y$. The quotient map $\pi_K: \mathcal{M} \cap \pi^{-1}_G(S) \to
S$ is a smooth submersion, hence $\tilde{\rho}$ is smooth along each stratum $S$. Over $S$, we can assume that $X
= Y = G/\Sigma \times X\hq G$. For $z_0 \in
\mathcal{M}$, we have 
\[T_{z_0}(\mathcal{M}) = T_{z_0}(K\acts z_0) \oplus
T_{z_0}(K^\C\acts z_0)^{\perp},\]
where $\perp$ denotes the perpendicular with respect to the Riemannian
metric associated to $\omega$. This allows us to construct a smooth
section $\sigma: X\hq G \rightarrow G/\Sigma \times X\hq G$ for $\pi$
with image in $\mathcal{M}$ in such a way that the
differential of $\sigma$ at $\pi(z_0)$ is a complex linear
isomorphism from $T_{\pi(z_0)}(X\hq G)$ to $ T_{z_0}(K^\C\acts
z_0)^\perp$. Let $\sigma(w)= (\eta(w), w)$. Let $\tilde f$ be a smooth
function near $\pi (z_0)$. Since $\rho$ is strictly
plurisubharmonic there exists an $\epsilon > 0$ such that
$\rho_\epsilon := \rho + \epsilon (\tilde f \circ \pi)$ is
plurisubharmonic on $X$. By construction, $\tilde
\rho_\epsilon := \tilde \rho + \epsilon \tilde f$ is equal to
$\rho_\epsilon \circ \sigma$. Now use the chain rule, the fact that
$\rho_\epsilon|_{K^\C\acts z}$ is critical at $K^\C\acts z \cap
\mathcal{M}$ and $d\eta(\pi(z_0))=0$ to show that
\[\frac{\partial^2\tilde\rho_\epsilon}{\partial
  w_i\partial\bar w_j}(\pi(z_0))= \frac{\partial^2\rho_\epsilon}{\partial
  w_i\partial\bar w_j}(z_0). \]
It follows from the considerations above that $\tilde \rho_\epsilon$ is plurisubharmonic on
each stratum $S \subset X\hq G$. Since $\tilde\rho_\epsilon$
extends continuously to $X\hq G$, the results of
\cite{Extensionofspshfunctions} imply that $\tilde\rho_\epsilon$ is
plurisubharmonic on $X\hq G$. Hence, $\tilde \rho$ is strictly
  plurisubharmonic.

Covering $X\hq G$ with sets $U_\alpha\hq G$, the
corresponding strictly plurisubharmonic functions $\tilde \rho_\alpha$ fit together to
define a K\"ahler structure on $X\hq G$ with the desired smoothness
properties. We emphasize at this point that the induced K\"ahler structure does not
depend on the choice of the pairs $(U_\alpha, \rho_\alpha)$.
\begin{rem}
Inspecting the proof that we have outlined above, one sees that K\"ahlerian reduction works under the following weaker regularity assumptions: as before, let $K$ be a compact Lie group and $G=K^\C$ its complexification. Let $X$ be a holomorphic K\"ahler $G$-space with analytic Hilbert quotient $\pi_G: X\rightarrow X\hq G$. Let $\mathbb{S}$ be a stratification of $X\hq G$ which is finer than the stratification of $X\hq G$ as a complex space and finer than the decomposition of $X\hq G$ by $G$-orbit types. Assume that the K\"ahler structure of $X$ is smooth along every $G$-orbit in $X$, on a Zariski-open smooth subset $X_{reg}$ of $X$ and along the set of closed orbits in $\pi_G^{-1}(S)$ for each stratum $S$ of $X\hq G$. Assume that the K\"ahler structure is $K$-invariant and that there exists a continuous map $\mu: X \rightarrow \mathfrak{k}^*$ which is a smooth momentum map for the $K$-action on $X_{reg}$ as well as on every $G$-orbit. In this situation, we call $X$ a \emph{stratified Hamiltonian K\"ahler $K^\C$-space}. Assume that $X=X(\mathcal{M})$. Then, the fundamental diagram \eqref{fundamental} holds (in particular, $X\hq G$ is homeomorphic to $\mu^{-1}(0)/K$) and the construction outlined above yields a K\"ahler structure on $X\hq G$ which is smooth along the stratification $\mathbb{S}$.
\end{rem}
The following example shows that even if the quotient is a smooth
manifold, we cannot expect the reduced K\"ahler structure to be
smooth. This illustrates that the reduction procedure is also sensible to
singularities of the map $\mathcal{M} \rightarrow
\mathcal{M}/K$. 
\begin{ex}
Consider $X = \mathbb{C}^2$ with the action of $\C^*= (S^1)^\C$ given by $t \acts
(z,w) = (tz,t^{-1}w)$. The standard K\"ahler form on $\C^2$ can be written
as $\frac{i}{2}\partial \bar \partial \rho$, where $\rho: v \mapsto
\|v\|^2$ denotes the square of the norm function associated to the standard
Hermitean product on $\C^2$. After identification of $Lie
(S^1)^*$ with $\R$, the momentum map associated to $\rho$ is given by $\mu(z,w) = |z|^2 - |w|^2$. It follows that the
momentum zero fibre is singular at the origin. Every point in $\C^2$
is semistable and the analytic Hilbert quotient is realised by $\pi:
\C^2 \rightarrow \C, (z,w) \mapsto zw$. Restriction of $\rho$ to
$\mu^{-1}(0)$ induces the function $\tilde \rho (\tilde{z}) = |\tilde{z}|$ on $\C^2\hq \C^* = \C$. It is continuous strictly
plurisubharmonic and smooth along the orbit type stratification of $\C$.
\end{ex}
\section{Reduction in steps}
From now on we consider the following situation: Let $K$ be a
connected compact
Lie group, $G = K^\C$ its complexification and let $X$ be a connected Hamiltonian
K\"ahler $K^\C$-manifold with $K$-equivariant
momentum map $\mu: X \rightarrow \mathfrak{k}^*$. Let $L$ be a closed normal subgroup of $K$. Then, $L^\C=: H$ is contained in $K^\C$ as a closed normal complex subgroup. The inclusion
$\iota: \mathfrak{l} \rightarrow \mathfrak{k}$ of $\mathfrak{l}= Lie (L)$ into $\mathfrak{k}$ induces an adjoint map $\iota^*:
\mathfrak{k}^* \rightarrow \mathfrak{l}^*$. The composition $\mu_L:
X\rightarrow \mathfrak{l}^*$, $\mu_L \definiere \iota^* \circ \mu$ is a momentum
map for the action of $L$ on $X$. Set $\mathcal{M}_L \definiere
\mu_L^{-1} (0)$. There is a corresponding set of semistable points $X(\mathcal{M}_L) \definiere \{x \in X \,|\, \overline{H \acts x}\cap \mathcal{M}_L \neq \emptyset \}$, and an analytic Hilbert quotient
$\pi_H : X(\mathcal{M}_L) \rightarrow X(\mathcal{M}_L)\hq H$. In the following we will investigate the
relations between this quotient and the quotient $\pi_G: X(\mathcal{M}) \to X(\mathcal{M})\hq G$. The main result we show here is
\begin{thm}[K\"ahlerian reduction in steps]\label{main}
With the notation introduced above, the following holds:
\begin{enumerate}
\item[a)] The analytic Hilbert quotient $Q_L\definiere X(\mathcal{M})\hq H$ exists and is realised as an open subset of
$X(\mathcal{M}_L)\hq H$. There is a holomorphic $G$-action on
$Q_L$ such that the quotient map $\pi_H:
X(\mathcal{M}) \rightarrow Q_L$ is $G$-equivariant.  The analytic Hilbert
quotient $\bar \pi: Q_l \rightarrow Q_L\hq G$ exists. It is naturally biholomorphic to $X(\mathcal{M})\hq G$ and the following diagram commutes
\[\begin{xymatrix}{
   X(\mathcal{M}_L)\ar[d]_{\pi_H} &\ar@{_{(}->}[l] X(\mathcal{M}) \ar[d]_{\pi_H}\ar[r]^{\pi_G}&  X(\mathcal{M})\hq G  \\
   X(\mathcal{M}_L)\hq H & \ar@{_{(}->}[l] Q_L \ar[ru]_{\bar \pi}   & .
}
  \end{xymatrix}
\]
\item[b)]The restriction of the momentum map $\mu$ to $\mathcal{M}_L$ is $L$-invariant and induces a momentum map for the $K$-action on $Q_L$. This makes $Q_L$ into a stratified Hamiltonian K\"ahler $K^\C$-space. The analytic Hilbert quotient $Q_L\hq G$ carries a K\"ahlerian structure induced by the K\"ahlerian structure of $Q_L$. This K\"ahlerian structure coincides with the K\"ahlerian structure obtained by reduction for the quotient $\pi_K:\mathcal{M}  \rightarrow \mathcal{M}/K \simeq X(\mathcal{M})\hq G$.
\item[c)]The $G$-orbit type stratification of $X(\mathcal{M})\hq G$ coincides with the $2$-step stratification which is defined in Section \ref{Kaehlerinsteps}.
\end{enumerate}
\end{thm}
\subsection{Analytic reduction in steps}\label{complex}
In this section we will prove part a) of Theorem \ref{main}.
\begin{thm}[Analytic reduction in steps]\label{normalquotient}
Let $X$ be a holomorphic $G$-space for the complex-reductive Lie group $G$ and let $H \vartriangleleft G$ be a reductive normal subgroup. Assume that the analytic Hilbert quotient $\pi_G: X \rightarrow X\hq G$ exists. Then, the analytic Hilbert
quotient $\pi_H : X \to X\hq H$ exists and admits a holomorphic $G$-action such that $\pi_H$ is $G$-equivariant. Furthermore, the analytic Hilbert
quotient of $X\hq H$ by $G$ exists and is naturally biholomorphic to $X\hq
G$. If $\bar \pi: X\hq H \rightarrow X\hq G$ denotes the quotient map,
the diagram 
\begin{align}\label{diagram}
\begin{xymatrix}{
X \ar[r]^{\pi_G}\ar[d]_{\pi_H}  & X\hq G \\
  X\hq H  \ar[ur]_{\bar \pi}&
}
\end{xymatrix}
\end{align}
commutes.
\end{thm}
\begin{proof}
The quotient map $\pi_G: X \rightarrow X\hq G$ is an $H$-invariant Stein map. Hence, the quotient $X\hq H$ exists (see \cite{SemistableQuotients}). Let $\pi_H: X
\rightarrow X\hq H$ be the quotient map. The map
\[\id_G \times \pi_H : G \times X \rightarrow
G \times X\hq H\]is an analytic Hilbert quotient for the
$H$-action on $G\times X$ which is given by the action of $H$ on the second factor. Since $H$ is a normal subgroup of $G$, the map
that is obtained by composition of the action map $G \times
X \rightarrow X$ with the quotient map
$\pi_H$ is an $H$-invariant holomorphic map from $ G \times
X$ to $X\hq H$. By the universal property of analytic Hilbert quotients we obtain a holomorphic
map  $ G \times
X\hq H \to X\hq H$ such that the following
diagram commutes:
\[\begin{xymatrix}{
G \times X \ar[r]\ar[d]_{\id_G \times\pi_H}&  X\ar[d]_{\pi_H}\\
G \times  X\hq H\ar[r] &  X\hq H.
}
\end{xymatrix}
\]
This defines a
holomorphic action of $G$ on $ X\hq H$ such that $\pi_H$
is $G$-equivariant.

The $G$-invariant map $\pi_G: X \to X\hq G$ descends to a $G$-invariant map $\bar \pi: X\hq H\rightarrow X\hq G$. We claim that $\bar \pi$ is Stein. Indeed, let $U \subset X\hq G$ be Stein. Since $\pi_G$ is an analytic Hilbert quotient map, the inverse image $\pi_G^{-1}(U)$ is a $\pi_H$-saturated Stein subset of $X$. Since
analytic Hilbert quotients of Stein spaces are Stein (see
\cite{HeinznerGIT}), $\pi_H(\pi_G^{-1}(U))= \bar\pi^{-1}(U)$ is a Stein open set
in $X\hq H$. Furthermore, using the fact that $\bar \pi$ is induced by
$\pi_G$ and that $\pi_H$ is $G$-equivariant, we see that
$\mathcal{O}_{X\hq G} = \bar \pi _* (\mathcal{O}_{X \hq H})^G$. This shows that the map $\bar \pi: X \hq H \rightarrow X\hq G$ is the analytic Hilbert
quotient of $X\hq H$ by the action of $G$.
\end{proof}
In the situation of Theorem \ref{main}, it follows from the previous theorem that the analytic Hilbert quotient of $X(\mathcal{M})$ by the action of a normal complex-reductive subgroup $H$ of $G$ exists. For our purposes, it is important to relate this quotient to the quotient of $X(\mathcal{M}_L)$ by $H$ and to the momentum geometry of $\mu_L$. For later reference, we make the following
\begin{defi}
 Let $X$ be a holomorphic $G$-space and $A \subset X$. We define $\mathcal{S}_G(A) \definiere \{x \in X \,
|\, \overline{G\acts x}\cap A \neq \emptyset\}$ and call it the
\emph{saturation} of $A$ with respect to $G$. If the space $X$ plays a
role in our considerations, we also write $\mathcal{S}_G^X(A)$.
\end{defi}

\begin{lemma}\label{saturated}
Let $L$ be a compact subgroup of $K$ and $H \definiere L^\C \subset G= K^\C$. Then, the set $X(\mathcal{M})$ is a $\pi_H$-saturated subset of $X(\mathcal{M}_L)$.
\end{lemma}
\begin{proof}
Let $x_0
\in X(\mathcal{M})$. By Proposition \ref{organisation} there exists a strictly
plurisubharmonic exhaustion function $\rho$ of the fibre
$F\definiere \pi_G^{-1}(\pi_G(x_0))$ with the property that $\mu|_F$ is
associated to $\rho$.
This implies that $\mathcal{M}\cap F$ is the set where $\rho|_F$
assumes its minimum. The restriction of $\rho$ to $C\definiere \overline{H\acts x_0} \cap X(\mathcal{M}) \subset F$ also is an exhaustion. This implies that $\rho|_C$ attains its minimum at some point $y_0 \in \overline{H\acts
x_0}$. For all $\xi \in \mathfrak{l}$, we have $\mu_L^\xi(y_0)=\left.\frac{d}{dt}\right|_{t=0}\rho(exp(it\xi)\acts
y_0)=0$. Hence, $y_0\in \mathcal{M}_L$ and $x_0 \in X(\mathcal{M}_L)$.

The first part of the proof shows that $X(\mathcal{M}) \subset \mathcal{S}_H(X(\mathcal{M}) \cap \mathcal{M}_L)$. Conversely, let $x \in \mathcal{S}_H(X(\mathcal{M}) \cap \mathcal{M}_L)$. Then, by definition, $\overline{H \acts x} \cap (X(\mathcal{M}) \cap \mathcal{M}_L) \neq \emptyset$. Since $X(\mathcal{M})$ is an open $H$-invariant neighbourhood of $X(\mathcal{M})\cap \mathcal{M}_L$, this implies $x \in X(\mathcal{M})$. Hence, we have shown that $X(\mathcal{M}) = \mathcal{S}_H(X(\mathcal{M}) \cap \mathcal{M}_L)$. This concludes the proof.
\end{proof}

As before, let $\pi_G:
X(\mathcal{M})\rightarrow X(\mathcal{M})\hq G$ and
$\pi_H:X(\mathcal{M}_L)\rightarrow X(\mathcal{M}_L)\hq H $ denote the quotient
maps. From Lemma \ref{saturated}, we obtain
\begin{cor}\label{realisationofquotient}
The analytic Hilbert quotient for the $H$-action on $X(\mathcal{M})$ is given by $\pi_H|_{X(\mathcal{M})}: X(\mathcal{M}) \rightarrow \pi_H(X(\mathcal{M})) \subset X(\mathcal{M}_L)\hq H$. In particular, $X(\mathcal{M})\hq H$ can be realised as an open subset of $X(\mathcal{M}_L)\hq H$.
\end{cor}

Theorem \ref{normalquotient} and Corollary \ref{realisationofquotient} prove part a) of Theorem \ref{main}.

As a preparation for the proof of part b) of Theorem \ref{main}, we will now relate the preceeding discussion to the actions of the groups $K$ and $L$.
\begin{lemma}\label{Kinvariant}
The momentum zero fibre $\mathcal{M}_L$ of $\mu_L$ as well as
the set of semistable points $X(\mathcal{M}_L)$ is $K$-invariant.
\end{lemma}
\begin{proof}
Let $x \in \mathcal{M}_L$. Since $L$ is normal in $K$, we have $\Ad(K)(\mathfrak{l})\subset
\mathfrak{l}$. This implies $\mu^\xi(k\acts x) = \Ad^*(k)(\mu(x))=0$. Hence, we have
$k\acts x \in \mathcal{M}_L$. 

Since $\mathcal{M}_L$ is $K$-invariant and $hH = Hk$ holds for all $k \in K$, we have $\overline{H \acts (k\acts x)}\cap \mathcal{M}_L = k\acts (\overline{H\acts x}\cap \mathcal{M}_L)$ for all $x \in X(\mathcal{M}_L)$. This shows the claim.
\end{proof}

Since $X(\mathcal{M}_L)$ is $K$-invariant, it
  follows by considerations analogous to those in the proof of Theorem \ref{normalquotient} that $X(\mathcal{M}_L)\hq H$ is a complex $K$-space. On
  $\pi_H(X(\mathcal{M}))$ the $K$-action coincides with the restriction of the $K^\C$-action to $K$.
However, we do not know if the action defined in this way on
$\pi_H(X(\mathcal{M}))\subset X(\mathcal{M}_L)\hq H$ extends to a
holomorphic action of $K^\C$ on $ X(\mathcal{M}_L)\hq H$ in general.

In two
  important special cases, there is a holomorphic $K^\C$-action on
  $X(\mathcal{M}_L)\hq H$. If $X(\mathcal{M}_L)\hq H$ is compact, its
  group of holomorphic automorphisms $\mathcal{A}$ is a complex Lie group. The
  action of $K$ on $X(\mathcal{M}_L)\hq H$ yields a homomorphism of
  $K$ into $\mathcal{A}$. This extends to a holomorphic homomorphism of $K^\C$
  into $\mathcal{A}$ by the universal property of $K^\C$, hence $K^\C$ acts
  holomorphically on $X(\mathcal{M}_L)\hq H$. The second case is the
  following: as we have seen in Lemma \ref{Kinvariant}, the complement of
  $X(\mathcal{M}_L)$ in $X$ is $K$-invariant. If it is an analytic subset of
  $X$, its $K$-invariance implies its $K^\C$-invariance. In this
  case, $K^\C$ acts on $X(\mathcal{M}_L)$ and hence on $X(\mathcal{M}_L)\hq H$.

\begin{lemma}\label{saturation}
We have $\mathcal{S}_G^{\pi_H(X(\mathcal{M}))}(\pi_H(\mathcal{M})) = \pi_H(X(\mathcal{M}))$.
\end{lemma}
\begin{proof}
Let $x \in X(\mathcal{M})$. Then, by definition, $\overline{G \acts x}
\cap \mathcal{M} \neq \emptyset$. This implies that $\overline{G \acts
\pi_H(x)}\supset \pi_H(\overline{G \acts x})$ intersects
$\pi_H(\mathcal{M})$ non-trivially. This shows the claim.
\end{proof}
\subsection{K\"ahler reduction in steps}\label{Kaehlerinsteps}
In this section we will prove part b) of Theorem \ref{main}. The
results of Section \ref{complex} show that it is sensible to restrict to the
situation where $X = X(\mathcal{M}) = X(\mathcal{M}_L)$ for the
discussion of K\"ahlerian reduction in this section and we will do
this from now on.

First we take a closer look at the compatibility of the
$G$-action on $X\hq H = X(\mathcal{M})\hq H$ with the orbit type stratification $\mathbb{S}^{X\hq
  H}_H$ of $X\hq H$. For later reference, we note the following
\begin{lemma}\label{subgroups}
Let $M$ be a Lie group with finitely many connected components and
$M_0$ a closed subgroup of $M$. Then, the following are equivalent:
\begin{enumerate}
\item $M$ and $M_0$ are isomorphic as topological groups,
\item $M_0 = M$.
\end{enumerate}
\end{lemma}
As a first application, we get
\begin{lemma}\label{invariance}
The stratification $\mathbb{S}^{X\hq H}_H$ is $G$-invariant.
\end{lemma} 
\begin{proof}
Let $z \in \mathcal{M}_L$ and consider the $H$-action on $G\acts z$. The orbit $H\acts z$ is closed. Since $H\acts g\acts z = g \acts H \acts z$ holds for all $g \in G$, all $H$-orbits in $G\acts z$ are closed (and have the same dimension). Since $G\acts z$ is connected, there is a principal $H$-stratum $S$ in $G\acts z$ and we may assume $z \in S$. For any $g\in G$ this yields $hH_zh^{-1} \subset H_{g\acts z}= g(H\cap G_z)g^{-1} = gH_zg^{-1}$ for some $h \in H$. Lemma \ref{subgroups} implies $h H_z h^{-1} = g H_zg^{-1} = H_{g\acts z}$ and therefore $G\acts z = S$.
\end{proof}
We now investigate the compatibility of the $K^\C$-action on $X\hq H$
with the induced K\"ahler structure.
\begin{prop}\label{Hamiltonian}
\begin{enumerate}
\item The reduced K\"ahler structure of $X\hq H$ is smooth along each $G$-orbit in $X\hq H$.
\item The reduced K\"ahler structure is $K$-invariant. The restriction of $\mu$ to $\mathcal{M}_L$ is $L$-invariant and induces a well-defined continuous map $\tilde{\mu}: X\hq H \rightarrow
\mathfrak{k}^*$ which is a smooth momentum map on each stratum of
$X\hq H$ as well as on
every $G$-orbit.
\end{enumerate}
\end{prop}
\begin{proof}
The induced K\"ahler
structure $\tilde{\omega}$ is smooth along the stratification
$\mathbb{S}^{X\hq H}_H$ by construction. Furthermore, Lemma \ref{invariance} shows that this stratification
is $G$-invariant. Hence, $\tilde{\omega}$ is smooth
along every $G$-orbit in $X\hq H$. 

Let $x \in \mathcal{M}_L \subset X$. Applying Propostion
\ref{organisation} to $X$, to the momentum map $\mu: X \rightarrow
\mathfrak{k}^*$ and to the quotient $\pi_G: X \to X\hq G$, we get a
$\pi_G$-saturated neighbourhood $U$ of $x$ in $X$ on which the K\"ahler
structure and the momentum maps $\mu$ and $\mu_L$ are induced by a $K$-invariant strictly
plurisubharmonic function $\rho: U \rightarrow \R$. Since the K\"ahler
structure on $\pi_H(U) \subset X\hq H$ is induced by the restriction
of $\rho$ to $\mathcal{M}_L \cap U$ it is $K$-invariant.

For the $L$-invariance of $\mu|_{\mathcal{M}_L}$ we follow \cite{StratifiedSymplectic}: first, we equivariantly identify $\k$ with $\mathfrak{k}^*$. The image of $\mu: \mathcal{M}_L
\rightarrow \mathfrak{k}$ is
contained in $\mathfrak{l}^\perp \subset \mathfrak{k}$. Here, $^\perp$ denotes the perpendicular with respect to a chosen $K$-invariant inner product on $\mathfrak{k}$. Hence, the
equation $\mu(l\acts x) = \Ad(l)(\mu(x))$ implies that it is sufficient
to show that $L$ acts trivially on $\mathfrak{l}^\perp$. As
$\mathfrak{l}^\perp$ is $L$-invariant, $[\mathfrak{l},\mathfrak{l}^\perp]$ is
contained in $\mathfrak{l}^\perp$. But $\mathfrak{l}$ is an ideal in
$\mathfrak{k}$ and hence, $[\mathfrak{l},\mathfrak{l}^\perp]$ is also contained
in $\mathfrak{l}$. This implies that the component of the identity $L_0$ of
$L$ acts trivially on $\mathfrak{l}^\perp$. Since $K$ is connected, the image of the map $\phi: K
\times L_0 \rightarrow K, (k,l) \mapsto klk^{-1}$ lies in some connected
component of $L$, which has to be $L_0$. Hence, $L_0$ is normal in
$K$. The finite group $L/L_0$ is normal in the connected group
$K/L_0$, hence central. It follows that $L/L_0$ acts trivially on $\Lie
(K/L_0)= \mathfrak{l}^\perp$. This shows that $L$ acts
trivially on $\mathfrak{l}^\perp$.

Hence, $\mu|_{\mathcal{M}_L}$ defines a continuous map $\tilde{\mu}$ on
$X\hq H$. This map is a smooth momentum map along each stratum $S$ of $X\hq H$, since $\pi_L: \mathcal{M}_L \cap \pi_H^{-1}(S) \rightarrow \tilde{S}$ is a smooth $K$-equivariant submersion. Since the stratification is $G$-invariant, $\tilde \mu$ is a smooth momentum map along each $G$-orbit in $X\hq H$.
\end{proof}
\begin{rem}
By construction, we see that $\tilde{\mathcal{M}}\definiere
\tilde{\mu}^{-1}(0)$ is equal to $\pi_H(\mathcal{M})$. Furthermore,
Lemma \ref{saturation} shows that $\mathcal{S}_G(\tilde{\mathcal{M}})=
X\hq H$.
\end{rem}
We will now investigate the possibility of carrying out K\"ahler reduction with respect to the quotient $\bar \pi: X\hq H \rightarrow X\hq G$.

First we define a second stratification of $X\hq G$ as follows: let $S$ be a stratum of the $H$-orbit type stratification of $X\hq H$ such that
$S\cap \tilde{\mathcal{M}} \neq \emptyset$. Stratify $\bar \pi (\bar
S) \setminus \bar\pi(\partial S)$ by $G$-orbit types with respect to the
$G$-action on $X\hq H$. Repeating this procedure for all strata yields a stratification of $X\hq G$ which we call the
\emph{$2$-step-stratification}. We will see later on that it coincides with the stratification of $X\hq G$ by $G$-orbit types (see Section \ref{stratificationinsteps}).

We notice that by Propostition \ref{Hamiltonian} and by the construction of the $2$-step stratification $X\hq H$ is a stratified Hamiltonian K\"ahler $K^\C$-space. We also note that by the fundamental diagram \eqref{fundamental}, $X\hq G$ is homeomorphic  to $\tilde{\mathcal{M}}/K$.  Applying the procedure described in Section \ref{KaehlerReduction}, we have
\begin{prop}
The analytic Hilbert quotient $X\hq G$ carries a K\"ahler structure induced from $X\hq H$ by K\"ahlerian reduction for the quotient $\tilde {\pi}_K: \tilde{ \mathcal{M}} \to \tilde{\mathcal{M}}/K$. This K\"ahler structure is smooth along the $2$-step stratification. 
\end{prop}

We now have two K\"ahlerian structures on the complex space $X\hq
G$: the structure $\omega _{red}$ that we get by reducing the K\"ahler form $\omega$ on $X$ to
$X\hq G$ and the K\"ahler structure $\hat \omega$ that we get by first reducing
$\omega$ to the K\"ahler structure $\tilde{\omega}$ on $X\hq H$ and
then reducing $\tilde{\omega}$ to $\hat \omega$ on $X\hq G$. In order to complete the proof of Theorem \ref{main} b), we have to prove that 
$\omega _{red}$ and $\hat \omega$ coincide.

Let $y_0 \in X\hq G$. Proposition \ref{organisation} yields a open neighbourhood $U$ of $y_0$ in $X\hq G$ such that on $\pi_G^{-1}(U)$ the K\"ahler structure and the momentum maps $\mu$ and $\mu_L$ are given by $\rho$. It follows that the K\"ahler structure on $\bar \pi^{-1}(U)$ is given by $\tilde \rho$ which is induced via $\rho|_{\mathcal{M}_L\cap \pi_G^{-1}(U)}$.

\begin{lemma}\label{potentialonfirstreduction}
$\tilde{\rho}$ is an exhaustion along every fibre of $\bar \pi$. On $\bar \pi^{-1}(U)$, we have $\tilde{\mu} =
  \mu_{\tilde{\rho}}$. 
\end{lemma}
\begin{proof}
Let $y \in U$ and let $F \definiere \pi_G^{-1}(y)$. We have $\mathcal{M} \cap F = K \acts x$ for some $x \in F$. The restriction $\rho|_{F}$ is an exhaustion. It is minimal along $\mathcal{M} \cap F = K \acts x$. Since $\tilde{F} \definiere \bar \pi^{-1}(y) = \pi_H(F)$ and since $\tilde{\rho}|_{\tilde{F}}$ is induced via the restriction of $\rho$ to $\mathcal{M}_L\cap F$, $\tilde {\rho}|_{\tilde{F}}$ is minimal along $\tilde{\mathcal{M}} \cap \tilde{F} = K\acts \pi_H(x)$. It follows that $\tilde{\rho}|_{\tilde{F}}$ is an exhaustion.

Since both $\tilde{\mu}$ and $\mu_{\tilde{\rho}}$ are momentum maps
for the $K$-action on $\bar \pi^{-1}(U)$, they differ by a constant $c \in
\mathfrak{k}^*$. Hence, it suffices to show that there exists an $z
\in \bar \pi^{-1}(U)$ such that $\tilde{\mu} (z) = \mu_{\tilde{\rho}}(z)$
holds. Let $z_0 \in
\tilde{\mathcal{M}}\cap \bar \pi^{-1}(U)$. As we have seen above, $\tilde {\rho}|_{\bar \pi^{-1}(\bar \pi (z_0))}$ is critical along $K\acts z_0$. Therefore, $\mu_{\tilde{\rho}}(z_0) =
\tilde{\mu}(z_0)=0$ and hence $c = 0$.
\end{proof}

The equality of the two K\"ahler structures on $U$ now follows from the construction of reduced structures:
the K\"ahler structure $\omega_{red}$ is formed by a single
function $\rho_{red}$ which is induced by the restriction of $\rho$ to
$\mathcal{M}\cap \pi_G^{-1}(U)$. The K\"ahler structure $\tilde{\omega}$ on $\bar \pi^{-1}(U)$ is given by $\tilde{\rho}$. Lemma \ref{potentialonfirstreduction} shows that the K\"ahler structure $\hat
\omega$ is given by a single function $\hat \rho$ that is induced by
the restriction of $\tilde{\rho}$ to $\tilde{\mathcal{M}}\cap \bar\pi^{-1}(U)$. However, $\tilde{\mathcal{M}}
= \mathcal{M}/L$ and therefore, $\hat \rho $ coincides with
$\rho_{red}$ after applying the homeomorphism $(\mathcal{M}/L)/K \simeq \mathcal{M}/K$. This concludes the proof of Theorem \ref{main}.
\subsection{Stratifications in steps}\label{stratificationinsteps}
Here we prove part c) of Theorem \ref{main}. We formulate it as
\begin{prop}\label{stratificationscoincide}
The stratification of $X\hq G$ by $G$-orbit types coincides with the
$2$-step-stra\-ti\-fi\-ca\-tion of $X\hq G$.
\end{prop}
\begin{proof}
We do this in two steps: first, we claim that the $G$-stratification
of $X\hq G$ is finer than the $2$-step-stratification. Let $S$
be a stratum of the $G$-stratification of $X\hq G$ and let $Y$ be the set
of closed orbits in $\pi_G^{-1}(S)$. Then, since $S$ and $G$ are
connected, $Y$ is a connected holomorphic
$H$-manifold.

Every $H$-orbit in $Y$ is closed. Stratify $Y$ by $H$-orbit types. Since $Y$
is connected, there exists a principal orbit type $(H_0)$ for the
$H$-action on $Y$. Then, $H_0= G_0 \cap H$, where $(G_0)$ is the orbit
type of the $G$-stratum $S$. Let $y \in Y$. We have $H_y = g H_0 g^{-1}$ for some $g \in G$. Since $(H_0)$ is the
principal isotropy type, there exists an $h \in H$ such that $h H_0 h^{-1} < H_y$. Lemma \ref{subgroups} implies that $H_y$ is conjugate to $H_0$ in $H$. Hence,
$\pi_H(Y)$ is a union of closed $G$-orbits that is contained in a
single stratum $\mathbb{S}_H^{X\hq H}$. Furthermore, the $G$-isotropy
group of each element in $\pi_H(Y)$ is conjugate in $G$ to $G_0H$. This implies that the
$G$-stratification of $X\hq G$ is finer than the
$2$-step-stratification.

In the second step, we prove that the $2$-step-stra\-ti\-fi\-cation of $X\hq
G$ is finer than the $G$-stra\-ti\-fi\-ca\-tion of $X\hq G$. Let $S$ be a
stratum of the $2$-step stratification of $X\hq G$. Let $w, w' \in \pi_G^{-1}(S)
\cap \mathcal {M}$. Then, by construction of $S$, we have 
\begin{enumerate}
\item \label{eins}$H_w$ is conjugate to $H_{w'}$ in $H$ and 
\item\label{zwei}$HG_w$ is conjugate to $HG_{w'}$ in $G$.
\end{enumerate}
We claim that $\dim G_w = \dim
G_{w'}$. Indeed, 1) and 2) imply the following dimension equalities on
the Lie algebra level:
\begin{align*}
\dim \mathfrak{g}_w &= \dim (\mathfrak{g}_w + \mathfrak{h})-\dim
\mathfrak{h} + \dim (\mathfrak{g}_w \cap \mathfrak{h})\\
&= \dim (\mathfrak{g}_{w'} + \mathfrak{h})-\dim
\mathfrak{h} + \dim (\mathfrak{g}_{w'} \cap \mathfrak{h})\\
&= \dim \mathfrak{g}_{w'}
\end{align*}
Hence, $\dim G\acts w = \dim G \acts {w'} =: m_0$. The set
\[Y\definiere  \{x \in \pi_G^{-1}(S)\,|\, G\acts x \text{ closed in }X \}=
\pi_G^{-1}(S)\cap \{x \in X \,|\, \dim G\acts x \leq m_0 \}\]is
analytic in $\pi_G^{-1}(S)$. There exists a principal orbit type $(G_0)$ for the action of
$G$ on $Y$. Let $w\in Y$ such that $(G_w)=(G_0)$. Let $w'$ be any other point in $Y$. 
Since $(G_w)$ is the principal orbit type for the action on $Y$, there
exists a $g_0\in G$ such that $g_oG_w g_0^{-1} \subset G_{w'}$. Hence,
we can assume that $G_w \subset G_{w'}$. Let $\iota: G_w \rightarrow
G_{w'}$ be the inclusion. Since $G_w \subset G_{w'}$, we also have
$H_w \subset H_{w'}$. Together with assumption \ref{eins}.) and Lemma
\ref{subgroups}, this implies that $H_w = H_{w'}$. If we let
$H_w=H_{w'}$ act on $G_w$ and $G_{w'}$ as a normal
subgroup, $\iota$ is clearly equivariant. This shows that $\iota$ induces an
injective group homomorphism $\tilde{\iota}: G_w/H_w \rightarrow
G_{w'}/H_w$. For every $y
\in Y$ we have the following exact
sequence
\[1 \rightarrow H_y \rightarrow G_y \rightarrow G_yH/H \rightarrow
1.\] Together with assumption \ref{zwei}.), this implies that $G_w/H_w$ and $G_{w'}/H_w$ are
isomorphic as topological groups. Since both $G_w/H_w$ and $G_{w'}/H_w$ have a finite
number of connected components, Lemma
\ref{subgroups} implies that $\tilde{\iota}$ is surjective. It follows
that $\iota$ is
surjective and that $G_{w'}$ is equal to $G_w$. This shows that the $2$-step stratification is finer than the
$G$-stratification of $X\hq G$ and we are done.
\end{proof}
\section{Projectivity and K\"ahlerian reduction}
One source of examples of quotients for actions of complex-reductive Lie groups is Geometric Invariant Theory (GIT). In this section we will explain the construction of GIT-quotients in detail, noting that it is an example of the K\"ahlerian reduction in steps procedure yielding projective algebraic quotient spaces. In addition we will discuss projectivity results for general Hamiltonian actions on complex algebraic varieties and their compatibility with reduction in steps.

A complex-reductive group $G=K^\C$ carries a uniquely determined algebraic structure making it into a
linear algebraic group. In this section we study algebraic actions of complex-reductive groups $G$ with respect to this algebraic structure. For this, we define a \emph{complex algebraic $G$-variety $X$} to be a complex algebraic variety $X$ together with an action of $G$ such that the action map $G\times X \to X$ is algebraic. 

Let $X$ be a projective complex algebraic $G$-variety. Let $L$ be a very ample $G$-linearised line
bundle on $X$, i.e. a very ample line bundle $L$ on $X$ and an action of $G$ on $L$ by bundle
automorphisms making the bundle projection $G$-equivariant. The action of $G$ on $X$ and on $L$ induces a natural representation on $V\definiere \Gamma(X, L)^*$ and there exists a $G$-equivariant embedding of $X$ into $\P(V)$. Let
\[\mathcal{N}(V) = \{v \in V \, | \, f(v)= 0 \text{ for all } f \in
\C[V]^G\}\]be the \emph{nullcone} of $V$. Here $\C[V]^G$ denotes the algebra of polynomials on $V$ that are invariant under the action of $G$. Let $X(V) \definiere X \setminus p(\mathcal{N}(V))$, where $p:
V\setminus \{0\} \rightarrow \P(V)$ denotes the projection. Then, it
is proven in Geometric Invariant Theory that the
analytic Hilbert quotient $\pi_G: X(V) \to X(V)\hq G$ exists, that
$X(V)\hq G$ is a projective
algebraic variety, that the quotient map is algebraic and affine, and that the
algebraic structure sheaf on the quotient is the sheaf of invariant
polynomials on $X(V)$ (see \cite{MumfordGIT}). In this situation, we call $\pi$ an \emph{algebraic Hilbert quotient}. We will explain the relation of this algebraic
quotient theory to reduction in steps. Consider the action of the complex-reductive group
$G \times \C^*$ on $V$ that is given by the $G$-representation and
the action of $\C^*$ by multiplication. By a theorem of
Hilbert, the analytic Hilbert quotient $V\hq G$ exists as an affine
algebraic variety. It can be embedded into a $\C^*$-representation
space $W$ with $\C[W]^{\C^*} = \C$ as a $\C^*$-invariant algebraic subvariety. The orbit space
$W\setminus\{0\}/\C^*$ is a weighted projective space. It
follows that the analytic Hilbert quotient of $V\setminus
\mathcal{N}(V)$ by the action of $G \times \C^*$ exists as a projective algebraic variety. Let $C(X)\definiere \overline{\pi^{-1}(X)}\subset V$ be
the \emph{cone over $X$}. The considerations above imply that the
analytic Hilbert quotient of $C(X) \setminus \mathcal{N}(V)$ by the
action of $G\times \C^*$ exists as a projective algebraic variety. Now we note that $X(V)=
(C(X)\setminus  \mathcal{N}(V))/\C^*$, i.e. $X(V)$ is the quotient of
$ C(X)\setminus  \mathcal{N}(V)$ by the normal subgroup $\C^*$ of $G
\times \C^*$. By Theorem \ref{normalquotient}, the analytic Hilbert quotient $\pi_G: X(V) \to X(V)\hq G$
exists and the following diagram commutes:
\[\begin{xymatrix}{
C(X)\setminus  \mathcal{N}(V) \ar[r]\ar[d]^{p}&  X(V)\hq G\\
X(V)\ar[ru]_{\pi_G}&.
}
\end{xymatrix}
\]
The variety $C(X)$ is K\"ahler with K\"ahler structure given by the square of the norm function associated to a $K$-invariant Hermitean product $<\cdot, \cdot>$ on $V$ . Furthermore, the action of $K\times S^1$ is
Hamiltonian. A momentum map is given by
\[\mu^{(\xi,\sqrt{-1})}(v) = 2i<\xi v, v>  +\,\|v\|^2 - 1 \quad \forall \xi \in \mathfrak{k}.\]Here, $\sqrt{-1}$ is considered as an element of $\Lie (S^1)$. The set of semistable points coincides with
$C(X)\setminus \mathcal{N}(V)$. The set of semistable points for the $\C^*$-action and the restricted momentum map is $V\setminus \{0\}$. The Fubini-Study-metric on $X$ is obtained by K\"ahlerian reduction
for the quotient $p: C(X)\setminus\{0\} \rightarrow
X$ and the $K$-action on $X$ is Hamiltonian with respect to the Fubini-Study-metric with momentum map $\tilde{\mu}^\xi ([v]) = \frac{2i <\xi v,v>}{\|v\|^2}$. The set of semistable points with respect to $\tilde{\mu}$ coincides with $X(V)$.

We will now take a closer look at the K\"ahlerian structure that we get by
K\"ahlerian reduction of the Fubini-Study-metric to $X(V)\hq
G \simeq \tilde{\mu}^{-1}(0)/K$. We have already noted that $X(V)\hq G$ is projective algebraic. It carries an ample
line bundle $L_{red}$ such that there exists an $n_0 \in \N$ with $\pi_G^*(L_{red}) = H^{n_0}$, where $H$ denotes the restriction
of the hyperplane bundle to $X(V)$. With this notation, the reduced K\"ahler structure $\omega_{red}$ on
$X(V)\hq G$ fulfills 
\[n_0 \cdot c_1(\omega_{red}) = c_1(L_{red}) \in H^2(X,\R), \]where $c_1$ denotes
the first Chern class of $\omega_{red}$ and of
$L_{red}$, respectively. Hence, the cohomology class of the K\"ahler structure obtained by K\"ahlerian reduction of the Fubini-Study-metric lies in the real span of the ample cone of $X(V)\hq G$.

The discussion above shows that analytic Hilbert quotients obtained by  K\"ahlerian reduction of projective algebraic varieties are again projective algebraic if the $K$-actions under consideration are Hamiltonian with respect to the Fubini-Study-metric. For arbitrary K\"ahler forms on an algebraic variety, i.e.\ forms that are not the curvature form of a very ample line bundle, and  associated momentum maps there is no obvious relation to ample $G$-line bundles and the corresponding analytic Hilbert quotients of sets of semistable points are not a priori algebraic. 

The best algebraicity result for momentum map quotients known so far is proven in \cite{Doktorarbeit}:

\textit{Let $K$ be a compact Lie group and $G = K^\C$. Let $X$ be a smooth complex algebraic $G$-variety such that the
$K$-action on $X$ is Hamiltonian with respect to a $K$-invariant
K\"ahler form on $X$ with momentum map $\mu: X \to \mathfrak{k}^*$. Let $\mathcal{M} \definiere \mu^{-1}(0)$, and let $\pi_G: X(\mathcal{M}) \to X(\mathcal{M})\hq G$ denote the quotient map. If $\mathcal{M}$ is
compact, the following holds:
\begin{enumerate}
 \item The analytic Hilbert quotient $X(\mathcal{M})\hq G$ is (the complex space associated to) a projective
algebraic variety.
\item There exists a $G$-equivariant
biholomorphic map $\Phi$ from $X(\mathcal{M})$ to an algebraic $G$-variety
$Y$, there exists an algebraic Hilbert
quotient $p_G: Y \to Y\hq G$, and the map $\bar \Phi: X(\mathcal{M})\hq G \to Y\hq G$ that is induced by $\Phi$ is an isomorphism of algebraic varieties.
\item The algebraic $G$-variety $Y$ is uniquely determined up to $G$-equivariant isomorphism of algebraic varieties.
\end{enumerate}}

As before, let $L \vartriangleleft K$ be a
normal closed subgroup of $K$ and $H = L^\C$. Then, reduction in steps is compatible with the algebraicity results obtained above in the following way:

We have already seen that the analytic Hilbert quotient $\pi_H: X(\mathcal{M}) \to X(\mathcal{M})\hq H$ exists. Furthermore, there exists an algebraic Hilbert
quotient $p_H: Y \rightarrow Y \hq H$ for the action of $H$ on $Y$ and the map $\Phi$ induces a $G$-equivariant biholomorphic map $\Phi_H: X(\mathcal{M})\hq H \to Y\hq H$. In addition, the algebraic Hilbert quotient $(Y \hq H) \hq G$ exists and it is biregular to $Y\hq G$. If $\bar p: Y\hq H \to Y\hq G$ denotes the quotient map, the following diagram commutes: 
\begin{align*}
\begin{xymatrix}{
X(\mathcal{M})\hq H \ar[rd]^{\bar\pi}\ar@/_3pc/[ddd]^{\Phi_H}                      &					\\
X(\mathcal{M}) \ar[d]^\Phi \ar[r]^{\pi_G} \ar[u]_{\pi_H} & X(\mathcal{M})\hq G \ar[d]^{\bar \Phi} \\
 Y \ar[r]^{p_G}\ar[d]^{p_H}  & Y\hq G \\
  Y\hq H  \ar[ur]_{\bar p}&.
}
\end{xymatrix}
\end{align*}

\vspace{3,5cm}
\textbf{Authors' adress:
}\\
Institut und Fakult\"at f\"ur Mathematik\\
Ruhr-Universit\"at Bochum\\
Universit\"atsstrasse 150\\
44780 Bochum\\
Germany

\textbf{e-mail:}\\
daniel.greb@ruhr-uni-bochum.de\\
heinzner@cplx.ruhr-uni-bochum.de
\end{document}